\documentclass{amsart}
\usepackage{mathrsfs}
\usepackage{times}
\usepackage{amssymb,amsmath,amscd, amsthm,stmaryrd,verbatim, mathrsfs}

\usepackage{color}
\usepackage{}
\usepackage{bbding}
\usepackage{amsfonts}
\usepackage[colorlinks,linkcolor=blue,urlcolor=cyan,citecolor=green,pagebackref]{hyperref}
\usepackage{amssymb,amscd,colordvi,amsthm,verbatim,mathrsfs,latexsym,  graphicx}

\newtheorem{theorem}{Theorem}[section]
\newtheorem{Prop}{Proposition}[section]
\newtheorem{lemma}{Lemma}[section]
\newtheorem{corollary}{Corollary}[section]
\newtheorem{remark}{Remark}[section]
\newtheorem{definition}{Definition}[section]
\newtheorem{example}{Example}[section]
\numberwithin{equation}{section}

\newcommand{\bb}{\mathbb}

\newcommand{\frk}{\mathfrak}

\newcommand{\ra}{\rangle}

\newcommand{\ptl}{\partial}
\newcommand{\dlt}{\delta}

\newcommand{\lmd}{\lambda}

\newcommand{\mbb}{\mathbb}

\newcommand{\Ann}{\mathrm{Ann}}

\newcommand{\eq}{\begin{equation}}
\newcommand{\en}{\end{equation}}
\newcommand{\beqna}[1]{\begin{eqnarray}\label{#1}}
\newcommand{\eeqna}{\end{eqnarray}}
\newcommand{\beqn}[1]{\begin{equation}\label{#1}}
\newcommand{\eeqn}{\end{equation}}

\numberwithin{equation}{section}

\usepackage{hyperref}
\begin{document}

\title[Gelfand-Kirillov dimensions of representations of   $\mathfrak{sl}(n)$]{ Gelfand-Kirillov Dimensions of  the $\bb{Z}^2$-graded
Oscillator Representations of   $\mathfrak{sl}(n)$ }

\author{Zhanqiang Bai}

\address{Institute of Mathematics, Academy of Mathematics and System Sciences, Chinese Academy of Sciences, Beijing 100080, China}
\email{bzq@amss.ac.cn}





\maketitle

\begin{abstract}
 We find an exact formula of Gelfand-Kirillov dimensions
  for the infinite-dimensional explicit irreducible $\mathfrak{sl}(n, \bb F)$-modules that appeared in the
  $\bb{Z}^2$-graded oscillator generalizations of the classical theorem on harmonic polynomials established by Luo
  and Xu.  Three infinite subfamilies of these modules
have the minimal Gelfand-Kirillov dimension. They contain weight
modules with unbounded weight multiplicities and completely pointed
modules.


\end{abstract}

\section{Introduction}
In 1960s, Gelfand-Kirillov \cite{Ge-Ki} introduced a quantity to
measure the rate of growth of an algebra in terms of any generating
set, which is now known as  Gelfand-Kirillov dimension. In 1970s,
Vogan \cite{Vogan-78} and Joseph \cite{Joseph-78} used the
Gelfand-Kirillov dimension to measure the size of the
infinite-dimensional modules. Gelfand-Kirillov dimension  has been
an important invariant in the theory of algebras over a field for
the past forty years. Although it is rarely exact, it has the
advantage over Krull dimension of being both symmetric and ideal
invariant. It has been applied successfully to enveloping algebras,
Weyl algebras, and more generally to filtered and graded algebras.
But in general, the Gelfand-Kirillov dimension of an
infinite-dimensional module is not easy to compute.

A module of a finite-dimensional simple Lie algebra is called a {\it
weight module} if it is a direct sum of its weight subspaces. Let
$M$ be an irreducible highest-weight module for a finite-dimensional
simple Lie algebra $\mathfrak{g}$. Then $M$ is naturally a weight
module with finite-dimensional weight subspaces. Denote by $d_{M}$
its Gelfand-Kirillov dimension. We fix a Cartan subalgebra
$\mathfrak{h}$, a root system $\Delta\subset \mathfrak{h}^*$  and a
set of positive roots $\Delta_+ \subset \Delta$. Let $\rho$ be half
the sum of all positive roots. Suppose that $\beta$ is the highest
root. It is well known that $d_M=0$ if and only if $M$ is
finite-dimensional, in which case irreducible modules are classified
by the highest-weight theory. From Vogan \cite{Vogan-81} and Wang
\cite{wang-99}, we know that the next smallest integer occurring is
$d_M=(\rho,\beta^{\vee})$. We call them the \emph{minimal
Gelfand-Kirillov dimension module}. These small modules are of great
interest in representation theory. A general introduction can be
found in Vogan \cite{Vogan-81}. Fernando \cite{Fer} showed that the
only simple Lie algebras which have simple torsion-free modules (or
cuspidal modules, i.e.,weight modules on which all root vectors of
the Lie algebra act bijectively)  are those of types $A_n$ or $C_n$.
He also showed that these finitely generated torsion-free $\mathcal
{U}(\mathfrak{g})$-modules have the minimal Gelfand-Kirillov
dimension. A similar result was independently obtained by Futorny
\cite{Fu}. In the last 20 years cuspidal modules have been
extensively studied by Benkart, Britten, Grantcharov, Hooper,
Khomenko, Lemire, Mathieu, Mazorchuk, Serganova and others (see e.g.
\cite{Br-Ho-Le},\cite{Br-Kh-Le}, \cite{Gr-Se-06},\cite{Gr-Se-10},
\cite{Ma}).

Benkart, Britten and Lemire \cite{Be-Br-Le} generalized Fernando's
work. They proved that infinite-dimensional weight modules, whose
weight-subspace dimensions are bounded, have the minimal
Gelfand-Kirillov dimension and exist only for finite-dimensional
simple Lie algebras of types $A_{n}$  and  $C_n$. Britten and Lemire
\cite{Br-Le} described all weights $\omega \in \mathfrak{h}^*$ such
that the corresponding irreducible highest-weight
$\mathfrak{sp}(2n,\mathbb{C})$-module $L(\omega)$ has bounded
multiplicities. Sun \cite{Sun} classified genuine irreducible
 lowest-weight modules of $\mathfrak{sp}(2n,
\mathbb{C})$ with minimal Gelfand-Kirillov dimension. But  other examples of minimal Gelfand-Kirillov dimension modules are not well-known.

In classical harmonic analysis,  a  fundamental theorem  says that
the spaces of homogeneous harmonic polynomials are irreducible
modules of the corresponding orthogonal Lie group (algebra) and the
whole polynomial algebra is a free module over the invariant
polynomials generated by harmonic polynomials. Bases of these
irreducible modules can be obtained easily (e.g., cf. \cite{Xu-08}).
The algebraic beauty of the above theorem is that Laplace equation
characterizes the irreducible submodules of the polynomial algebra
and the corresponding quadratic invariant gives a decomposition of
the polynomial algebra into a direct sum of irreducible submodules,
namely, the complete reducibility. Algebraically, this can be
interpreted as an $(\mathfrak{sl}(2,\mbb{R}),\mathfrak{o}(n,\mbb{R}))$ Howe duality.
Recently Luo and Xu \cite{Luo-Xu} established the $\bb{Z}^2$-graded
oscillator generalizations of the above theorem for
$\mathfrak{sl}(n,\bb F)$, where the irreducible submodules are $\bb
Z^2$-graded homogeneous polynomial solutions of deformed Laplace
equations. An important feature of the irreducible modules in
\cite{Luo-Xu} is that the corresponding representation formulas are
simple and bases are easily given. In fact, these are explicit
infinite-dimensional highest-weight irreducible $\mathfrak{sl}(n,\bb
F)$-modules. The goal of this paper is to find an exact formula of
Gelfand-Kirillov dimension for them. It turns out that these
Gelfand-Kirillov dimensions are independent of the double grading
and four infinite subfamilies of these modules have the minimal
Gelfand-Kirillov dimension. In general, our result will be useful in
studying the $\mathfrak{sl}(n,\bb F)$-modules with a given
Gelfand-Kirillov dimension. Below we give more detailed technical
introduction.

Throughout this paper, the base field $\bb F$ has characteristic
$0$. For convenience, we will use the notion
$\overline{{i,i+j}}=\{i,i+1,i+2,...,i+j\}$ for integers $i$ and $j$
with $i\leq j$. Denote by $\bb{N}$ the additive semigroup of
nonnegative integers. Let $E_{r,s}$ be the square matrix with $1$ as
its $(r,s)$-entry and $0$ as the others. Moreover, we always assume
that $n\geq 2$ is an integer.  Denote
${\mathscr{B}}=\bb{F}[x_1,...,x_n,y_1,...,y_n]$. Fix
 $n_1,n_2\in\overline{1,n}$ with $n_1\leq n_2$. Changing operators $\partial_{x_r}\mapsto -x_r,\;
 x_r\mapsto
\partial_{x_r}$  for $r\in\overline{1,n_1}$ and $\partial_{y_s}\mapsto -y_s,\;
 y_s\mapsto\partial_{y_s}$  for $s\in\overline{n_2+1,n}$ in the canonical oscillator representation $E_{i,j}|_{\mathscr{B}}=x_{i}\partial_{x_j}-y_{j}\partial_{y_i}$ for $i,j \in \overline{1,n}$, we get the following non-canonical oscillator representation of $\mathfrak{sl}(n,\bb{F})$ on
  ${\mathscr{B}}$ determined by
\eq E_{i,j}|_{\mathscr{B}}=E_{i,j}^x-E_{j,i}^y\qquad \text{for}~ i,j\in\overline{1,n}\en with
\eq E_{i,j}^x|_{\mathscr{B}}=\left\{\begin{array}{ll}-x_j\partial_{x_i}-\delta_{i,j}&\mbox{if}\;
i,j\in\overline{1,n_1};\\ \partial_{x_i}\partial_{x_j}&\mbox{if}\;i\in\overline{1,n_1},\;j\in\overline{n_1+1,n};\\
-x_ix_j &\mbox{if}\;i\in\overline{n_1+1,n},\;j\in\overline{1,n_1};\\
x_i\partial_{x_j}&\mbox{if}\;i,j\in\overline{n_1+1,n}
\end{array}\right.\en
and
\eq E_{i,j}^y|_{\mathscr{B}}=\left\{\begin{array}{ll}y_i\partial_{y_j}&\mbox{if}\;
i,j\in\overline{1,n_2};\\ -y_iy_j&\mbox{if}\;i\in\overline{1,n_2},\;j\in\overline{n_2+1,n};\\
\partial_{y_i}\partial_{y_j} &\mbox{if}\;i\in\overline{n_2+1,n},\;j\in\overline{1,n_2};\\
-y_j\partial_{y_i}-\delta_{i,j}&\mbox{if}\;i,j\in\overline{n_2+1,n}.
\end{array}\right.\en
The related variated  Laplace operator becomes
\eq {\mathscr{D}}=-\sum_{i=1}^{n_1}x_i\partial_{y_i}+\sum_{r=n_1+1}^{n_2}\partial_{x_r}\partial_{y_r}-\sum_{s=n_2+1}^n
y_s\partial_{x_s}\en and its dual
\eq\eta=\sum_{i=1}^{n_1}y_i\partial_{x_i}+\sum_{r=n_1+1}^{n_2}x_ry_r+\sum_{s=n_2+1}^n
x_s\partial_{y_s}.\en
 Set
\eq{\mathscr{B}}_{\langle
\ell_1,\ell_2\rangle}=\mbox{Span}\{x^\alpha
y^\beta\mid\alpha,\beta\in\bb{N}\:^n,\sum_{r=n_1+1}^n\alpha_r-\sum_{i=1}^{n_1}\alpha_i=\ell_1,
\sum_{i=1}^{n_2}\beta_i-\sum_{r=n_2+1}^n\beta_r=\ell_2\}\en for
$\ell_1,\ell_2\in\bb{Z}$. Define
\eq{\mathscr{H}}_{\langle\ell_1,\ell_2\rangle}=\{f\in
{\mathscr{B}}_{\langle \ell_1,\ell_2\rangle}\mid
{\mathscr{D}}(f)=0\}.\en Luo and Xu \cite{Luo-Xu} proved that
 for any $\ell_1,\ell_2\in\bb{Z}$ such that
$\ell_1+\ell_2\leq n_1-n_2+1-\delta_{n_1,n_2}$,
${\mathscr{H}}_{\langle \ell_1,\ell_2 \rangle}$ is an irreducible
highest-weight $\mathfrak{sl}(n,\bb{F})$-module. Moreover, the
homogeneous subspace $\mathscr{
B}_{\langle\ell_1,\ell_2\rangle}=\bigoplus_{m=0}^\infty\eta^m(\mathscr{
H}_{\langle\ell_1-m,\ell_2-m\rangle})$ is a direct sum of
irreducible submodules. In some special cases, they obtained more
general results. The following is the main theorem of this paper.

\begin{theorem}\label{main} For any $\ell_1,\ell_2\in\bb{Z}$, if
$\mathfrak{sl}(n,\bb{F})$-module  ${\mathscr{H}}_{\langle
\ell_1,\ell_2 \rangle}$ is irreducible, then it has the
Gelfand-Kirillov dimension
\eq \label{formula}
d=\left\{
    \begin{array}{ll}
      2n-2, &  \hbox{\text{\emph{if}} $1<n_1< n_2<n-1$ \emph{or} $3\leq n_1=n_2\leq n-3, n\geq 7$;} \\
      2n-3, & \hbox{\emph{if} $1=n_1< n_2<n$ \emph{or} $n_1<n_2=n-1$ \emph{or} $n_1=n_2=3<n=6$;} \\
       2n-4, & \hbox{if $n_1=n_2=2<n-1 \mathrm{~or~} 1<n_1=n_2=n-2$;}\\
       n, & \hbox{\emph{if} $1<n_1=n_2<n-1$;} \\
       n-1, & \hbox{\emph{otherwise}.}
    \end{array}
  \right. \en
\end{theorem}

We find that some of these irreducible highest-weight $\mathfrak{sl}(n,\bb{F})$-modules have the minimal Gelfand-Kirillov dimension. The result is as follows.

\begin{corollary} An irreducible
$\mathfrak{sl}(n,\bb{F})$-module ${\mathscr{H}}_{\langle
\ell_1,\ell_2 \rangle}$  has the minimal Gelfand-Kirillov dimension $n-1$
if and only if \begin{enumerate}
                        \item[(1)] $n_1< n_2=n$.
                           \item[(2)] $n_1=n_2=1$.
                           \item[(3)] $n_1=n_2=n-1$.
                         \end{enumerate}

\end{corollary}

From Benkart-Britten-Lemire \cite{Be-Br-Le},  an infinite-dimensional weight module is called \emph{completely pointed} if its weight spaces are all one-dimensional. Using the results on modules for Weyl algebras, they provided the whole list of simple infinite-dimensional completely pointed modules. These modules are described in terms of multiplication and differentiation operators on ``polynomials". Comparing with their result \cite{Be-Br-Le}, we can give  different realizations for some of these simple infinite-dimensional completely pointed modules.

\begin{corollary} An irreducible
$\mathfrak{sl}(n,\bb{F})$-module ${\mathscr{H}}_{\langle
\ell_1,\ell_2 \rangle}$ (with highest weight vector $v_{\lambda}$ of weight $\lambda$)  is completely pointed
if and only if \begin{enumerate}
    \item[(1)] $n_1+1<n_2=n$,\\
                            ${\mathscr{H}}_{\langle
\ell_1,\ell_2 \rangle}={\mathscr{H}}_{\langle
m_1,0 \rangle}$,  with $v_{\lambda}=x_{n_1+1}^{m_1}$, $\lambda=-(m_1+1)\lambda_{n_1}+m_1\lambda_{n_1+1}$,
$m_1\in\bb{N}$ \\
or ${\mathscr{H}}_{\langle
\ell_1,\ell_2 \rangle}={\mathscr{H}}_{\langle
-m_1,0 \rangle}$,  with $v_{\lambda}=x_{n_1}^{m_1}$, $\lambda=m_1\lambda_{n_1-1}-(m_1+1)\lambda_{n_1}$, $m_1\leq n-n_1-2$ or $m_{1}\geq n-n_1-1$.

                           \item[(2)] $n_1+1=n_2=n$,\\
                           ${\mathscr{H}}_{\langle
\ell_1,\ell_2 \rangle}={\mathscr{H}}_{\langle
-m_1,m_2 \rangle}$,  with $v_{\lambda}=x_{1}^{m_1}y_{2}^{m_{2}}$, $\lambda=(m_2-m_1-1)\lambda_{1}$,
$m_i\in\bb{N}$ and $m_2\leq m_1$, $n=2$ \\
or  ${\mathscr{H}}_{\langle
\ell_1,\ell_2 \rangle}={\mathscr{H}}_{\langle
-m_1,0 \rangle}$,  with $v_{\lambda}=x_{n-1}^{m_1}$, $\lambda=m_1\lambda_{n-2}-(m_1+1)\lambda_{n-1}$,
$m_1\in\bb{N}$ \\
or ${\mathscr{H}}_{\langle
\ell_1,\ell_2 \rangle}={\mathscr{H}}_{\langle
m,0 \rangle}$,  with $v_{\lambda}=x_{n-1}^{m}$, $\lambda=m\lambda_{n-2}-(m+1)\lambda_{n-1}$,
$m\in\bb{Z}$.

                           \item[(3)] $n_1=n_2=1$,\\
                           ${\mathscr{H}}_{\langle
\ell_1,\ell_2 \rangle}={\mathscr{H}}_{\langle
-m_1,0 \rangle}$,  with $v_{\lambda}=x_{1}^{m_1}$, $\lambda=-(m_1+2)\lambda_{1}$,
$m_1\in\bb{N}$\\
or  ${\mathscr{H}}_{\langle
\ell_1,\ell_2 \rangle}={\mathscr{H}}_{\langle
m_1+1,-m_1-1 \rangle}$,  with $v_{\lambda}=\zeta_{2}^{m_1+1}$, $\lambda=-(m_1+3)\lambda_{1}$,
$m_1\in\bb{N}$, $n=3$.
                           \item[(4)] $n_1=n_2=n-1$,\\ ${\mathscr{H}}_{\langle
\ell_1,\ell_2 \rangle}={\mathscr{H}}_{\langle
0,-m_2 \rangle}$,  with $v_{\lambda}=y_{n}^{m_2}$, $\lambda=-(m_2+2)\lambda_{n-1}$,
$m_2\in\bb{N}$\\
or ${\mathscr{H}}_{\langle
\ell_1,\ell_2 \rangle}={\mathscr{H}}_{\langle
-m_2-1,m_2+1 \rangle}$,  with $v_{\lambda}=\zeta_{1}^{m_2+1}$, $\lambda=-(m_2+3)\lambda_{2}$,
$m_2\in\bb{N}$, $n=3$.

                         \end{enumerate}

\end{corollary}

Benkart-Britten-Lemire \cite{Be-Br-Le} showed that when infinite-dimensional simple $\mathfrak{sl}(n, \bb F)$-modules have bounded weight multiplicities,  they must  have the minimal GK-dimension $n-1$. From our formula and Luo-Xu \cite{Luo-Xu}, we can easily find that the converse doesn't hold. For example, when $n_1=n_2=1<n$, the  irreducible highest-weight module ${\mathscr{H}}_{\langle-1, -1\rangle}$ has  a highest weight vector
$x_{1}y_{2}$ of weight
$-4\lmd_{1}+\lmd_{2}$. But it doesn't have bounded weight multiplicities.

In Section 2 we recall the basic definition of the Gelfand-Kirillov dimension. In Section 3, we give the formula on the Gelfand-Kirillov dimensions of these highest-weight modules ${\mathscr{H}}_{\langle \ell_1,\ell_2 \rangle}$. In Section 4, we give a proof for our formula. The proof is given in a case-by-case way. In Section 5, we give a corollary about the $\mathfrak{sl}(n,\bb{F})$-modules which have the minimal Gelfand-Kirillov dimension.

\section{Preliminaries on Gelfand-Kirillov Dimension}
In this section we recall the  definition and some properties of the Gelfand-Kirillov dimension.  The details can be found in Refs.\cite{Bo-Kr, Ja, Kr-Le, NOTYK, Vogan-78, Vogan-91}.

\begin{definition}
Let $A$ be an algebra (not necessarily associative) generated  by a
finite-dimensional subspace $V$. Let $V^n$ denote the linear span of
all products of length   at most $n$ in elements of $V$. The
\emph{Gelfand-Kirillov dimension} of $A$  is defined by:
$$GKdim(A) =  \limsup_{n\rightarrow \infty} \frac{\log\mathrm{dim}( V^{n} )}{\log n}$$
\end{definition}

\begin{remark}It is well-known that the above definition  is independent of  the choice of the
finite  dimensional  generating subspace $V$ (see Ref.\cite{Bo-Kr, Kr-Le}).
Clearly $GKdim(A)=0$ if and only if $\mathrm{dim}(A) <\infty$.
\end{remark}
The notion of Gelfand-Kirillov dimension can be extended for left $A$-modules. In fact, we have the following definition.
\begin{definition} Let $A$ be  an algebra (not necessarily associative) generated  by a finite-dimensional subspace $V$. Let $M$ be a left $A$-module generated  by a finite-dimensional subspace $M_{0}$.  Let $V^n$ denote the linear span of all products of length  at most $n$  in elements of $V$.
The \emph{Gelfand-Kirillov dimension} $GKdim(M)$ of $M$  is defined by
$$GKdim(M) = \limsup_{n\rightarrow \infty}\frac{\log\mathrm{dim}( V^{n}M_{0} )}{\log n}.$$
\end{definition}

In particular,  let $\frk{g}$ be a complex Lie algebra. Let $A=\mathcal{U}(\frk g)$ be the enveloping algebra of $\frk g$, with the standard filtration given by $A_{n}=\mathcal{U}_{n}(\frk g)$, the subspace of $\mathcal{U}(\frk g)$ spanned by products of at most $n$-elements of $\frk g$. By the Poincar\'{e}-Birkhoff-Witt theorem (see Knapp \cite[Prop. 3.16]{Knapp}), the graded algebra $\text{gr} (\mathcal{U}(\frk g))$ is canonically isomorphic to the symmetric algebra $S(\frk g)$. Suppose $M$ is a $\mathcal{U}(\frk g)$-module generated  by a finite-dimensional subspace $M_{0}$. We set $M_{n}=\mathcal{U}_{n}(\frk g)M_{0}$.  Denote $\text{gr}M=\bigoplus\limits_{n=0}^{\infty}\text{gr}_{n}M$, where $\text{gr}_{n}M=M_{n}/M_{n-1}$. Then $\text{gr}M$ becomes a graded $S(\frk g)$-module. We denote $\dim(M_{n})$ by $\varphi_{M}(n)$. Then we have the following lemma.

\begin{lemma}\label{hi-se}
\emph{(Hilbert-Serre \cite[Chapter VII. Th.41]{Za-Sa} and \cite{Ya-94})}
\begin{enumerate}
  \item[(1)] With the notations as above, there exists a unique polynomial $\tilde{\varphi}_{M}(n)$ such that $\varphi_{M}(n)=\tilde{\varphi}_{M}(n)$ for large $n$. The leading term of $\tilde{\varphi}_{M}(n)$ is $$\frac{c(M)}{(d_{M})!}n^{d_{M}},$$ where $c(M)$ is an integer.
  \item[(2)]The   degree $d_{M}$ of this polynomial $\tilde{\varphi}_{M}(n)$ is equal to  the dimension of the \emph{associated variety} $$\mathscr{V}(M)=\{X\in \frk{g}^{*}|p(X)=0, \forall  p\in \Ann_{S(\frk{g})}(\text{\emph{gr}}M)\},$$
where $\Ann_{S(\frk{g})}(\text{\emph{gr}}M)=\{D\in S(\frk{g})| Dv=0, \forall  v\in\text{\emph{gr}}M\}$
 is the  annihilator ideal of $\mathrm{gr}M$ in $S(\frk{g})$, and $S(\frk{g})$ is identified with the polynomial ring over $\frk{g}^{*}$ through the Killing form of $\frk g$.
\end{enumerate}
\end{lemma}
\begin{remark}

From the definition of Gelfand-Kirillov dimension, we know $$GKdim(M)=\limsup_{n\rightarrow \infty} \frac{\log\dim (U_{n}(\frk g)M_{0} )}{\log n}=\limsup_{n\rightarrow \infty} \frac{\log\tilde{\varphi}_{M}(n)}{\log n}=d_{M}=\dim \mathscr{V}(M).$$
\end{remark}

\begin{example}
Let $M=\mathbb{C}[x_{1},...,x_{k}]$. Then $M$ is an algebra generated by the  finite-dimensional subspace $V=Span_{\mathbb{C}}\{x_{1},...,x_{k}\}$.  So $M_{n}=V^{n}=\bigoplus\limits_{0\leq q \leq n}P_{q}[x_{1},...,x_{k}]$ is the subset of homogeneous polynomials of degree $\leq n$.
 Then \begin{align*}\varphi_{M}(n)=&\sum\limits_{0\leq q \leq n}\dim_{\mathbb{C}}(P_{q}[x_{1},...,x_{k}])\\
=&\sum\limits_{0\leq q\leq n}\binom{k+q-1}{q}\\
=&\binom{k+n}{n}\\
=&\frac{n^k}{k!}+O(n^{k-1}).\end{align*}
Then we have $GKdim(M)=k.$
\end{example}


\section{A formula on the Gelfand-Kirillov dimension}
We keep the same notations with the introduction. Write
\eq\zeta_1=x_{n_1-1}y_{n_1}-x_{n_1}y_{n_1-1},\;\;\zeta_2=x_{n_1+1}y_{n_1+2}-x_{n_1+2}y_{n_1+1}.
\en
From Luo-Xu \cite{Luo-Xu}, the followings are detailed irreducible highest weight
$\mathfrak{sl}(n,\bb{F})$-modules:
\begin{enumerate}
  \item[(1)] $n_1+1<n_2<n$. For $m_1,m_2\in\bb{N}$ with $m_1+m_2\geq n_2-n_1-1$, ${\mathscr{H}
}_{\langle-m_1,-m_2\rangle}$ is an infinite-dimensional irreducible
module with  a highest-weight vector $x_{n_1}^{m_1}y_{n_2+1}^{m_2}$.
For $m_1,m_2\in\bb{N}$ with $m_2-m_1\geq n_2-n_1-1$,
${\mathscr{H}}_{\langle m_1,-m_2\rangle}$ is an infinite-dimensional
irreducible module with a highest-weight vector
$x_{n_1+1}^{m_1}y_{n_2+1}^{m_2}$. For $m_1,m_2\in\bb{N}$ with
$m_1-m_2\geq n_2-n_1-1$, ${\mathscr{H}}_{\langle -m_1,m_2\rangle}$
is an infinite-dimensional irreducible module with a highest-weight
vector $x_{n_1}^{m_1}y_{n_2}^{m_2}$.
  \item[(2)] $n_1+1<n_2=n$. For $m_1,m_2\in\bb{N}$,  ${\mathscr{H}}_{\langle
m_1,m_2\rangle}$  is an infinite-dimensional irreducible module with
a highest-weight vector $x_{n_1+1}^{m_1}y_n^{m_2}$. For
$m_1,m_2\in\bb{N}$ with $m_1\leq n-n_1-2$ or $m_2-m_1\leq n_1-n+1$,
${\mathscr{H}}_{\langle -m_1,m_2\rangle}$ is an infinite-dimensional
irreducible module with a highest-weight vector
$x_{n_1}^{m_1}y_n^{m_2}$.
  \item[(3)]  $n_1+1=n_2<n$. For
$m_1,m_2\in\bb{N}$, ${\mathscr{H}}_{\langle-m_1,-m_2\rangle}$ is an
infinite-dimensional irreducible module with  a highest-weight
vector $x_{n_1}^{m_1}y_{n_1+2}^{m_2}$. For $m_1,m_2\in\bb{N}$ with
$m_2-m_1\geq 0$, ${\mathscr{H}}_{\langle m_1,-m_2\rangle}$ is an
infinite-dimensional irreducible module with a highest-weight vector
$x_{n_1+1}^{m_1}y_{n_1+2}^{m_2}$. For $m_1,m_2\in\bb{N}$ with
$m_1-m_2\geq 0$, ${\mathscr{H}}_{\langle -m_1,m_2\rangle}$  is an
infinite-dimensional irreducible module with a highest-weight vector
$x_{n_1}^{m_1}y_{n_1+1}^{m_2}$.
\item[(4)]  $n_1+1=n_2=n$. For $m_1,m_2\in\bb{N}$ with $m_2\leq m_1$, ${\mathscr{H}}_{\langle -m_1,m_2\rangle}$ is an infinite-dimensional irreducible module with a highest-weight vector
$x_{n-1}^{m_1}y_n^{m_2}$.  Moreover,
 ${\mathscr{H}}_{\langle
m,0\rangle}$ is an infinite-dimensional irreducible module with a
highest-weight vector $x_{n-1}^m$.  For $m_1,m_2\in\bb{N}+1$,
${\mathscr{H}}_{\langle m_1,m_2\rangle}$  is an infinite-dimensional
irreducible module with a highest-weight vector\\
$\eta^{m_1+m_2}(x_{n-1}^{m_2}y_n^{-m_1})$.
\item[(5)] $n_1=n_2<n-1$. Let $m_1,m_2\in\mbb{N}$. The subspace ${\mathscr{H}}_{\langle-m_1,-m_2\ra}$ is an infinite-dimensional irreducible module with  a highest-weight vector
$x_{n_1}^{m_1}y_{n_1+1}^{m_2}$. The subspace ${\mathscr{H}}_{\langle
m_1+1,-m_2-m_1-1\ra}$ is an infinite-dimensional irreducible module
with  a highest-weight vector $y_{n_1+1}^{m_2}\zeta_2^{m_1+1}$. If
$n_1\geq 2$, the subspace
${\mathscr{H}}_{\langle-m_1-m_2-1,m_2+1\ra}$ is an
infinite-dimensional irreducible module with a highest-weight vector
$x_{n_1}^{m_1}\zeta_1^{m_2+1}$.
\item[(6)] $n_1=n_2=n-1$. Let $m_1,m_2\in\mbb{N}$. The subspace ${\mathscr{H}}_{\langle-m_1,-m_2\ra}$ is an infinite-dimensional irreducible module with a
highest-weight vector $x_{n_1}^{m_1}y_{n_1+1}^{m_2}$.  If $n\geq 3$,
the subspace \\
${\mathscr{H}}_{\langle-m_1-m_2-1,m_2+1\ra}$ is an
infinite-dimensional irreducible module with  a highest-weight
vector $x_{n_1}^{m_1}\zeta_1^{m_2+1}$.
\item[(7)] $n_1=n_2=n$. Let $m_1,m_2\in\mbb{N}$. The subspace ${\mathscr{H}}_{\langle-m_1-m_2,m_2\ra}$ is finite-dimensional and is an irreducible module with a
highest-weight vector $x_n^{m_1}\zeta_1^{m_2}$.
\end{enumerate}

Then we have our main theorem, i.e. Theorem \ref{main}.


\section{Proof of the main theorem}
Now we want to compute the Gelfand-Kirillov dimension of  the $\mathfrak{sl}(n,\bb{F})$-module ${\mathscr{H}}_{\langle \ell_1,\ell_2 \rangle}$ for all cases. Suppose its highest weight vector is $v_\lambda$ with highest weight $\lambda$.

We simply write  $E_{i,j}|_{\mathscr{B}}$ as $E_{i,j}$.
Take
\eq \mathfrak{h}=\sum_{i=1}^{n-1}\mbb{F}(E_{i,i}-E_{i+1,i+1})\en
as a Cartan subalgebra of $\mathfrak{sl}(n,\mbb{F})$ and the subspace spanned
by positive root vectors:
\eq \mathfrak{sl}(n,\mbb{F})_+=\sum_{1\leq i<j\leq n}\mbb{F}E_{i,j}.\en
Correspondingly, we have \eq \mathfrak{sl}(n,\mbb{F})_-=\sum_{1\leq i<j\leq n}\mbb{F}E_{j,i}.\en

From the highest-weight module theorem we know that ${\mathscr{H}}_{\langle \ell_1,\ell_2 \rangle}=\mathcal {U}(\mathfrak{g})v_{\lambda}=\mathcal {U}(\mathfrak{g_{-}})v_{\lambda}$. In the following we will compute the Gelfand-Kirillov dimension of $\mathcal {U}(\mathfrak{g_{-}})v_{\lambda}$ in  a case-by-case way.

Firstly we need the following two well-known lemmas.
\begin{lemma}\emph{(Multinomial theorem)}\\
Let $n,m$ be two positive integers, then
\eq \left|\{(k_1,k_2,...,k_m)\in \mathbb{N}^{m}|\sum\limits_{i=1}^{m}k_{i}=n\}\right|={n+m-1 \choose m-1}.
\en
\end{lemma}

\begin{lemma}Let $p, n$ be two positive integers, then
\eq \sum\limits_{i=0}^{n}i^p=\frac{(n+1)^{p+1}}{p+1}+\sum\limits_{k=1}^{p}\frac{B_k}{p-k+1}{p \choose k}(n+1)^{p-k+1}\approx \frac{n^{p+1}}{p+1},\en
where $B_k$ denotes a Bernoulli number.
\end{lemma}

From these two lemmas, we can get the following several propositions.

\begin{Prop}\label{ak}Let $k\in \mathbb{N}$ and we denote $M_{k}=\left\{\prod\limits_{\substack{1\leq i\leq n_1  \\ n_{1}+1\leq t\leq n}} ({x_i x_{t}})^{p_{it}}| \sum\limits_{\substack{1\leq i\leq n_1  \\ n_{1}+1\leq t\leq n}} p_{it}=k, p_{it}\in\mathbb{N} \right\}$. Then \eq a_k=\dim Span_{\mathbb{R}}M_{k}={n_{1}+k-1 \choose k}{n-n_{1}+k-1 \choose k}\approx ak^{n-2},\en for some constant \emph{a}.
\end{Prop}
\begin{proof}From the definition of $M_k$, we know that all the elements in $M_k$ are monomials and they must form a basis for $Span_{\mathbb{R}}M_{k}$. Thus
\begin{align*}a_k=&\dim Span_{\mathbb{R}}M_{k}=\#\left\{\prod\limits_{\substack{1\leq i\leq n_1  \\ n_{1}+1\leq t\leq n}} ({x_i x_{t}})^{p_{it}}| \sum\limits_{\substack{1\leq i\leq n_1  \\ n_{1}+1\leq t\leq n}} p_{it}=k, p_{it}\in\mathbb{N} \right\}\\
=&\#\left\{\prod\limits_{\substack{1\leq i\leq n_1 }} (x_i)^{\sum_{n_{1}+1\leq t\leq n}p_{it}}\prod\limits_{\substack{n_{1}+1\leq t\leq n}} (x_t)^{\sum_{1\leq i\leq n_1 }p_{it}}| \sum\limits_{\substack{1\leq i\leq n_1  \\ n_{1}+1\leq t\leq n}} p_{it}=k, p_{it}\in\mathbb{N} \right\}\\
=&{n_{1}+k-1 \choose k}{n-n_{1}+k-1 \choose k}\approx ak^{n-2}, \emph{~~for some constant~} \emph{a}.
\end{align*}

\end{proof}

\begin{Prop}\label{Nk}
Let $k\in \mathbb{N}$ and we denote $N_k=\left\{\prod\limits_{\substack{1\leq i\leq n_1  \\ n_{1}+1\leq t\leq n}} ({x_i x_{t}}-y_{i}y_{t})^{h_{it}}| \sum\limits_{\substack{1\leq i\leq n_1  \\ n_{1}+1\leq t\leq n}} h_{it}=k, h_{it}\in\mathbb{N} \right\}$.
Then we have \begin{align*} d_k=\dim Span_{\mathbb{R}}N_{k}\approx\left\{
                                                   \begin{array}{ll}
                                                     c_0k^{n-2}, & \text{if~} {n_1=1 \mathrm{~or~} n_1=n-1;} \\
                                                     c_1k^{2n-5}, & \text{if~} {n_1=2<n-1 \mathrm{~or~} 1<n_1=n-2;} \\
                                                     c_2k^{2n-4}, & \text{if~} {n_1=3<n=6;} \\
                                                    c_3k^{2n-3}, & \text{if~} {3\leq n_1\leq n-3, n\geq 7.}
                                                   \end{array}
                                           \right. \end{align*}
Here $c_0, c_1,c_2$ and $c_3$  are some positive constants which are independent of $k$.
\end{Prop}
\begin{proof}
When $n_1=1$, we have$$d_k=\dim Span_{\mathbb{R}}N_{k}=\left\{\prod\limits_{\substack{ 2\leq t\leq n}} ({x_1 x_{t}}-y_{1}y_{t})^{h_{t}}| \sum\limits_{\substack{2\leq t\leq n}} h_{t}=k, h_{t}\in\mathbb{N} \right\}={n-1+k-1 \choose k}\approx c_0k^{n-2},$$ for some positive constant $c_0$. The case for $n_1=n-1$ is dual to the previous case.

When $n_1=2<n-1$, we have
\begin{align*}d_k&=\dim Span_{\mathbb{R}}N_{k}\\
&=\dim Span_{\mathbb{R}}\left\{\prod\limits_{\substack{1\leq i\leq 2  \\ 3\leq t\leq n}} ({x_i x_{t}}-y_{i}y_{t})^{h_{it}}| \sum\limits_{\substack{1\leq i\leq 2  \\ 3\leq t\leq n}} h_{it}=k, h_{it}\in\mathbb{N} \right\}\\
&\geq \dim Span_{\mathbb{R}}\left\{\prod\limits_{3\leq t\leq n}(x_1x_t)^{h_{1t}}\prod\limits_{3\leq t\leq n}(y_2 y_t)^{h_{2t}}| \sum\limits_{\substack{1\leq i\leq 2  \\ 3\leq t\leq n}} h_{it}=k, h_{it}\in\mathbb{N} \right\}\\
&\approx c_{11}k^{2n-5}.
\end{align*}
On the other hand, we have $d_k=\dim Span_{\mathbb{R}}N_{k}\leq c_{12}k^{n_1(n-n_1)-1}=c_{12}k^{2n-5},$ for some positive constant $c_{12}$. So we must have $d_k=\dim Span_{\mathbb{R}}N_{k}\approx c_{1}k^{2n-5},$ for some positive constant $c_{1}$. The case for $1<n_1=n-2$ is dual to the previous case.

When $n_1=3<n=6$, we have \begin{align*}d_k&=\dim Span_{\mathbb{R}}N_{k}\\
&=\dim Span_{\mathbb{R}}\left\{\prod\limits_{\substack{1\leq i\leq 3  \\ 4\leq t\leq 6}} ({x_i x_{t}}-y_{i}y_{t})^{h_{it}}| \sum\limits_{\substack{1\leq i\leq 3  \\ 4\leq t\leq 6}} h_{it}=k, h_{it}\in\mathbb{N} \right\}\\
&\geq \dim Span_{\mathbb{R}}\left\{({x_1 x_{4}})^{h_{14}}({x_1 x_{5}})^{h_{15}}({x_2 x_{6}})^{h_{26}}({x_2 x_{5}})^{h_{25}}({x_3 x_{6}})^{h_{36}}\right.\\
&\quad\quad\quad\quad \quad \quad\quad\left.\cdot({y_1 y_{6}})^{h_{16}}({y_2 y_{4}})^{h_{24}}({y_3 y_{4}})^{h_{34}}({y_3 y_{5}})^{h_{35}}| \sum\limits_{\substack{1\leq i\leq 3  \\ 4\leq t\leq n}} h_{it}=k, h_{it}\in\mathbb{N} \right\}\\
&=\dim Span_{\mathbb{R}}\left\{\left( ({x_1})^{h_{14}+h_{15}}({x_{2}})^{h_{25}+h_{26}}(x_3)^{ h_{36}}y_{4}^{h_{24}+h_{34}}y_{5}^{h_{35}}y_{6}^{h_{16}}\right) \right.  \\
& \quad\quad\quad\quad\quad  \cdot \left. \left( ({x_4})^{h_{14}}({x_{5}})^{h_{15}+h_{25}}(x_6)^{h_{26}+h_{36}}(y_2)^{ h_{24}}(y_3)^{ h_{34}+h_{35}}y_{1}^{h_{16}}\right)
| \sum\limits_{\substack{1\leq i\leq 3  \\ 4\leq t\leq n}} h_{it}=k, h_{it}\in\mathbb{N} \right\}\\
&\approx c_{21}k^{8}.
\end{align*}
On the other hand, we have $d_k=\dim Span_{\mathbb{R}}N_{k}\leq c_{22}k^{n_1(n-n_1)-1}=c_{22}k^{8},$ for some positive constant $c_{22}$. So we must have $d_k=\dim Span_{\mathbb{R}}N_{k}\approx c_{2}k^{8}=c_{2}k^{2n-4},$ for some positive constant $c_{2}$.

When $3=n_1< n-3$, we have \begin{align*}d_k&=\dim Span_{\mathbb{R}}N_{k}\\
&=\dim Span_{\mathbb{R}}\left\{(\prod\limits_{\substack{1\leq i\leq 3  \\ 4\leq t\leq 6}} ({x_i x_{t}}-y_{i}y_{t})^{h_{it}})(\prod\limits_{\substack{1\leq i\leq 3  \\ 7\leq t\leq n}}(x_ix_t-y_iy_t)^{h_{it}})| \sum\limits_{\substack{1\leq i\leq 3  \\ 4\leq t\leq n}} h_{it}=k, h_{it}\in\mathbb{N} \right\}\\
&\geq \dim Span_{\mathbb{R}}\left\{(\prod\limits_{\substack{1\leq i\leq 3  \\ 4\leq t\leq 6}} ({x_i x_{t}}-y_{i}y_{t})^{h_{it}})(\prod\limits_{\substack{7\leq t\leq n}}({x_3 x_{t}})^{h_{3t}}({y_1 y_{t}})^{h_{1t}}(y_{2}y_t)^{h_{2t}})\right.\\
& \quad \quad \quad \quad \quad \quad \left.| \sum\limits_{\substack{1\leq i\leq 3  \\ 4\leq t\leq n}} h_{it}=k, h_{it}\in\mathbb{N} \right\}\\
&\geq \dim Span_{\mathbb{R}}\left\{\left( ({x_1})^{h_{14}+h_{15}}({x_{2}})^{h_{25}+h_{26}}(x_3)^{ h_{36}+\sum h_{3t}}y_{4}^{h_{24}+h_{34}}y_{5}^{h_{35}}y_{6}^{h_{16}}\prod\limits_{7\leq t \leq n} y_{t}^{h_{1t}+h_{2t}}\right) \right.  \\
& \quad\quad\quad\quad\quad  \cdot \left. \left( ({x_4})^{h_{14}}({x_{5}})^{h_{15}+h_{25}}(x_6)^{h_{26}+h_{36}}(\prod\limits_{7\leq t \leq n}x_{t}^{h_{3t}})(y_2)^{ h_{24}+\sum h_{2t}}(y_3)^{ h_{34}+h_{35}}y_{1}^{h_{16}+\sum h_{it}}\right)\right.\\
&\quad \quad \quad \quad \quad \quad \left.| \sum\limits_{\substack{1\leq i\leq 3  \\ 4\leq t\leq n}} h_{it}=k, h_{it}\in\mathbb{N} \right\}\\
&\approx c_{31}k^{2n-3}.
\end{align*}
On the other hand, we have
\begin{align*}d_k&=\dim Span_{\mathbb{R}}N_{k}\\
&\leq \dim Span_{\mathbb{R}} \left\{\prod\limits_{\substack{1\leq i\leq n_1  \\n_1+1 \leq t\leq n }} ({x_i x_{t}})^{p_{it}}
\prod\limits_{\substack{1\leq i\leq n_1  \\ n_{1}+1\leq t\leq n}}({y_{i} y_{t}})^{q_{it}} | \sum p_{it}+\sum q_{it}=k \right\}\\
&\approx c_{32}k^{2n-3},\end{align*} for some positive constant $c_{32}$. So we must have $d_k=\dim Span_{\mathbb{R}}N_{k}\approx c_{3}k^{2n-3},$ for some positive constant $c_{3}$. The case for $3<n_1=n-3$ is dual to the previous case.

When $3<n_1<n-3$, we can use the same inductive argument with the above case and get $$d_k=\dim Span_{\mathbb{R}}N_{k}\approx c_{3}k^{2n-3},$$ for some positive constant $c_{3}$.
\end{proof}

%
%
%

\begin{Prop}\label{d'k}

 Let $k\in \mathbb{N}$. Suppose $2<n_1+1\leq n_2<n-1$ and we denote
 \begin{align*} \nonumber  N^{\prime}_k&=\left\{\prod\limits(x_ix_s)^{p_{is}}\prod\limits(y_{s}y_{t})^{q_{st}}\prod\limits({x_{i} x_{t}}-y_{i}y_{t})^{h_{it}}|\right.\\
  &\quad\quad \left.\sum\limits_{\substack{1\leq i\leq n_1  \\n_1+1\leq s\leq n_2}}p_{is}+\sum\limits_{\substack{n_1+1\leq s\leq n_2 \\ n_{2}+1\leq t\leq n}}q_{st}+ \sum\limits_{\substack{1\leq i\leq n_1   \\ n_{2}+1\leq t\leq n}}h_{it}=k \right\},
  \end{align*}
 then
$$d'_k=\dim Span_{\mathbb{R}}N'_{k}\approx ck^{2n-3},$$ for some constant $c$.

\end{Prop}

\begin{proof}When $n_1=2<n_2<n-1$, then from Prop \ref{Nk} we have
\begin{align*}&d'_{k}=\dim Span_{\mathbb{R}}\left\{\prod\limits(x_ix_s)^{p_{is}}\prod\limits(y_{s}y_{t})^{q_{st}}\prod\limits({x_{i} x_{t}}-y_{i}y_{t})^{h_{it}}| \sum\limits p_{is}+\sum\limits q_{st}+ \sum\limits h_{it}=k \right\}\\
&\geq \dim Span_{\mathbb{R}} \left\{\left(({x_2})^{h_{2,n_2+1}+\sum p_{2s}}(x_1)^{\sum p_{1s}}(\prod\limits_{n_2+2\leq t\leq n}(x_1)^{ h_{1t}})(\prod\limits_{n_2+2\leq t\leq n}(y_t)^{h_{2t}}) (y_{t})^{\sum q_{st}}({y_{n_2+1}})^{h_{1,n_2+1}}\right)\right.   \\
& \quad \quad\quad\quad\quad\quad \left. \cdot \left( \prod\limits(x_s)^{p_{1s}+p_{2s}}(x_{n_2+1})^{h_{2,n_2+1}}(\prod\limits_{n_2+2\leq t\leq n}(x_t)^{h_{1t}})(\prod\limits(y_{s})^{q_{st}})({y_{1}})^{h_{1,n_2+1}}\prod\limits_{n_2+2\leq t\leq n}(y_2)^{ h_{2t}}\right)\right.\\
&\quad \quad\quad\quad\quad\quad \left. | \sum\limits p_{is}+\sum\limits q_{st}+ \sum\limits h_{it}=k \right\}\\
&\approx c_0k^{2n-3}, \emph{~for some constant }c_0.
\end{align*}

On the other hand, we have
\begin{align}\label{d'k<}\nonumber
d'_k=&\dim Span_{\mathbb{R}}N'_{k}\\\nonumber
\leq &\dim Span_{\mathbb{R}}\left\{\prod\limits ({x_i x_{s}})^{p_{is}}\prod({y_{s} y_{t}})^{q_{st}}\prod({x_{i} x_{t}})^{l_{it}}\prod(y_{i}y_{t})^{f_{it}} |\right.\\\nonumber
&\quad\quad\quad\quad\quad\left. \sum h_{it}+\sum p_{it}+\sum l_{it}+\sum f_{it}=k \right\}\\
=&\dim Span_{\mathbb{R}}\left\{\prod\limits_{\substack{1\leq i\leq n_1  \\n_1+1 \leq s_0\leq n }} ({x_i x_{s_0}})^{p_{is_0}}
\prod\limits_{\substack{1\leq i_0\leq n_2  \\ n_{2}+1\leq t\leq n}}({y_{i_0} y_{t}})^{q_{i_{0}t}} | \sum p_{is_0}+\sum q_{i_{0}t}=k \right\}\\\nonumber
\approx & c_{00} k^{2n-3},  \emph{~for some constant }c_{00}.
\end{align}

So we must have $d'_k=\dim Span_{\mathbb{R}}N'_{k}\approx ck^{2n-3},$ for some positive constant $c$.

From Prop \ref{Nk}, we can use the similar argument to compute the other cases. And for these cases we still have $d'_k=\dim Span_{\mathbb{R}}N'_{k}\approx ck^{2n-3},$ for some positive constant $c$.

\end{proof}

Next, we will compute the Gelfand-Kirillov dimensions of our modules in a case-by-case way.

\subsection{Case 1. $n_1+1<n_2<n$.}

In this case we have:
\begin{align} E_{r,i}|_{\mathscr{B}}&=-x_i\ptl_{x_r}-y_i\ptl_{y_r}\qquad  &\text{for~}& 1\leq i<r\leq
n_1,\\
 E_{s,i}|_{\mathscr{B}}&=-{x_i x_{s}}-{y_i}\ptl_{y_{s}}\qquad  &\text{for} ~&i\in\overline{1,n_1},\;s\in\overline{n_1+1,n_2},
\\
E_{t,i}|_{\mathscr{B}}&=-{x_i x_{t}}+y_{i}y_{t} \qquad  &\text{for} ~&i\in\overline{1,n_1},\;t\in\overline{n_2+1,n},\\
 E_{s,j}|_{\mathscr{B}}&=x_s\ptl_{x_j}-y_j\ptl_{y_s}\qquad   &\text{for}~& n_1<
j<s\leq n_2,\\
 E_{t,s}|_{\mathscr{B}}&=x_{t}\ptl_{x_{s}}+y_{s} y_{t}\qquad  &\text{for} ~&s\in\overline{n_1+1,n_2},\;t\in\overline{n_{2}+1,n},\\
 E_{t,p}|_{\mathscr{B}}&=x_t\ptl_{x_p}+y_t\ptl_{y_p}\qquad&\text{for}~& n_2+1\leq p<t\leq
n.
\end{align} Then the above $E_{j,i}$ forms a basis for the subalgebra $\mathfrak{sl}(n,\mbb{F})_-$.

From Luo-Xu \cite{Luo-Xu} we know that for $m_1,m_2\in\bb{N}$ with $m_1+m_2\geq n_2-n_1-1$, ${\mathscr{H}
}_{\langle-m_1,-m_2\rangle}$ has  a highest-weight vector
$v_{\lambda}=x_{n_1}^{m_1}y_{n_2+1}^{m_2}$ of weight
$\lambda=m_1\lambda_{n_1-1}-(m_1+1)\lambda_{n_1}-(m_2+1)\lambda_{n_2}+m_2(1-\delta_{n_2,n-1})\lambda_{n_2+1}$.
Then

\begin{align} E_{s,j}|_{\mathscr{B}}v_{\lambda}&=0   \qquad &\text{for}~&n_1<
j<s\leq n_2,\\
E_{n_1,i}|_{\mathscr{B}}v_{\lambda}&=-x_i\ptl_{x_{n_{1}}}(x_{n_1}^{m_1}y_{n_2+1}^{m_2}),\qquad &\text{for}~ &1\leq i<
n_1,\\
E_{r,i}|_{\mathscr{B}}v_{\lambda}&=0  \qquad &\text{for}~ &1\leq i<r<
n_1,\\
E_{t,p}|_{\mathscr{B}}v_{\lambda}&=0 \qquad &\text{for}~&n_2+1< p<t\leq
n,\\
E_{t,n_2+1}|_{\mathscr{B}}v_{\lambda}&=y_t\ptl_{y_{n_2+1}}(x_{n_1}^{m_1}y_{n_2+1}^{m_2}) \qquad &\text{for}~&n_2+1<t\leq
n.
\end{align}

Let $\mathfrak{g}_1$ be the subalgebra of $\mathfrak{sl}(n,\bb{F})$ spanned by the following set:$$\{E_{s,j},E_{r,i},E_{t,p}|n_1<
j<s\leq n_2, 1\leq i<r\leq
n_1, n_2+1\leq p<t\leq
n\}.$$

Let $\mathfrak{g}_2$ be the subalgebra of $\mathfrak{sl}(n,\bb{F})$ spanned by the following set:$$\{E_{s,i},E_{t,i},E_{t,s}|i\in\overline{1,n_1},s\in\overline{n_1+1,n_2},t\in\overline{n_{2}+1,n}\}.$$

So we get $\mathcal{U}(\mathfrak{g}_{-})=\mathcal{U}(\mathfrak{g}_{2})\mathcal{U}(\mathfrak{g}_{1}).$

Observe that \begin{align*}
&\mathcal {U}(\mathfrak{g}_1)v_{\lambda}\\
=&Span_{\mathbb{R}}\{\prod\limits_{i=1}^{n_{1}-1}E_{n_1,i}^{k_{i}}\prod\limits_{t=n_{2}+2}^{n}E_{t,n_{2}+1}^{l_{t}}v_{\lambda}|k_i, l_{t}\in \mbb N\}\\
=&Span_{\mathbb{R}}\{\prod\limits_{i=1}^{n_{1}-1}x_{i}^{k_{i}}x_{n_1}^{(m_{1}-\sum\limits_{i=1}^{n_{1}-1}k_{i})}
\prod\limits_{t=n_{2}+2}^{n}y_{t}^{l_{t}}y_{n_{2}+1}^{(m_{2}-\sum\limits_{t=n_{2}+2}^{n}l_{t})}|k_i, l_{t}\in \mbb N\}\\
=&Span_{\mathbb{R}}\{\prod\limits_{i=1}^{n_{1}}x_{i}^{k_{i}}\prod\limits_{t=n_{2}+1}^{n}y_{t}^{l_{t}}|
\sum\limits_{i=1}^{n_{1}}k_{i}=m_{1},\sum\limits_{t=n_{2}+1}^{n}l_{t}=m_{2}\}
\end{align*}

We denote this space by $M_{0}$.  Then $M_{0}$ is a subspace of the space spanned by the homogeneous polynomials of degree $m_1+m_2$ in $\mbb F[x_1,...,x_{n_1},y_{n_{2}+1},...,y_n]$. So $M_{0}$ is finite-dimensional.

Thus
$$\mathcal {U}(\mathfrak{g_{-}})v_{\lambda}=\mathcal {U}(\mathfrak{g}_{2})M_0.$$
Now we take any base element $u_{0}=\prod\limits_{i=1}^{n_{1}}x_{i}^{k_{i}}\prod\limits_{t=n_{2}+1}^{n}y_{t}^{l_{t}}\in M_0$.
Let $k$ be any positive integer.
We want to compute $\mathrm{dim}(\mathcal{U}_{k}(\mathfrak{g}_{2})M_{0})$, and get the Gelfand-Kirillov dimension of $\mathcal{U}(\mathfrak{g}_{2})M_{0}$.

We denote
\begin{align*}N_{0}(k)=&\left\{\left(\prod E_{s,i}^{p_{si}}\prod E_{t,i}^{h_{ti}}\prod E_{t,s}^{q_{ts}}\right)u_{0}|p_{si},h_{ti},q_{ts}\in \bb N, \right.\\
&\quad\quad\quad\quad\quad\quad\quad\quad\quad\quad\quad\quad\quad\left. \sum \limits_{\substack{1\leq i\leq n_1 \\ n_1+1\leq s \leq n_2 }}  p_{si}+\sum\limits_{\substack{1\leq i\leq n_1  \\ n_{2}+1\leq t\leq n}} h_{ti}+\sum\limits_{\substack{ n_1+1\leq s \leq n_2 \\ n_{2}+1\leq t\leq n}} q_{ts}=k\right\}.\end{align*}

From the definition we know
\begin{align*}&\left(\prod E_{s,i}^{p_{si}}\prod E_{t,i}^{h_{ti}}\prod E_{t,s}^{q_{ts}}\right)u_{0}\\
=&\left(\prod (-{x_i x_{s}}-{y_i}\ptl_{y_{s}})^{p_{si}}\prod(-{x_{i} x_{t}}+y_{i}y_{t})^{h_{ti}} \prod(y_{s}y_{t})^{q_{ts}}\right)\prod\limits_{i=1}^{n_{1}}x_{i}^{k_{i}}\prod\limits_{t=n_{2}+1}^{n}y_{t}^{l_{t}}\\
=&\left(\prod(-{x_i x_{s}})^{p_{si}}\prod(-{x_{i} x_{t}}+y_{i}y_{t})^{h_{ti}} \prod(y_{s}y_{t})^{q_{ts}}\right)\prod\limits_{i=1}^{n_{1}}x_{i}^{k_{i}}\prod\limits_{t=n_{2}+1}^{n}y_{t}^{l_{t}}\\
&+\text{lower degree part of }~y_{s}.\\
\end{align*}

Then we must have $$\dim Span_{\mathbb{R}}N_{0}(m)\geq d'_{m}.$$

Using the same idea with inequality \ref{d'k<}, we can also get $$\dim Span_{\mathbb{R}}N_{0}(m)\leq d'_{m}.$$

Thus $\dim Span_{\mathbb{R}}N_{0}(m)= d'_{m}.$

Then using proposition \ref{d'k}, we can get
\begin{align*}
&\dim Span_{\mathbb{R}}(\bigcup\limits_{0\leq m \leq k}N_{0}(m))\\
=& \sum\limits_{0\leq m \leq k} d'_{m}\\
=&\left\{
    \begin{array}{ll}
      b_1\sum\limits_{0\leq m\leq k}m^{2n-4}, & \hbox{if $2=n_1+1<n_2<n$ ~or~ $n_1+1<n_2=n-1$}\\
      c_1\sum\limits_{0\leq m\leq k}m^{2n-3}, & \hbox{if $2<n_1+1<n_2<n-1$}
    \end{array}
  \right.\\
=& \left\{
    \begin{array}{ll}
      bk^{2n-3}, & \hbox{if $2=n_1+1<n_2<n$ ~or~ $n_1+1<n_2=n-1$}\\
      ck^{2n-2}, & \hbox{if $2<n_1+1<n_2<n-1$.}
    \end{array}
  \right.
\end{align*}

We know $$ \dim Span_{\mathbb{R}}(\bigcup\limits_{0\leq m\leq k }N_{0}(m))\leq \mathrm{dim}(\mathcal{U}_{k}(\mathfrak{g}_{2})M_{0})\leq \dim{M_0}\dim Span_{\mathbb{R}}(\bigcup\limits_{0\leq m\leq k }N_{0}(m)).$$

%

Then from the definition, we know that the Gelfand-Kirillov dimension of $\mathcal{U}(\mathfrak{g}_{-})v_{\lambda}$ is
$$d=\left\{
    \begin{array}{ll}
      2n-3, & \hbox{if $2=n_1+1<n_2<n$ ~or~ $n_1+1<n_2=n-1$}\\
      2n-2, & \hbox{if $2<n_1+1<n_2<n-1$.}
    \end{array}
  \right.$$

From Luo-Xu \cite{Luo-Xu} we know that for $m_1,m_2\in\bb{N}$ with $m_2-m_1\geq n_2-n_1-1$, ${\mathscr{H}}_{\langle m_1,-m_2\rangle}$ has a highest-weight vector
$x_{n_1+1}^{m_1}y_{n_2+1}^{m_2}$ of weight
$-(m_1+1)\lambda_{n_1}+m_1\lambda_{n_1+1}-(m_2+1)\lambda_{n_2}+m_2(1-\delta_{n_2,n-1})\lambda_{n_2+1}$.
For $m_1,m_2\in\bb{N}$ with $m_1-m_2\geq n_2-n_1-1$, ${\mathscr{H}}_{\langle
-m_1,m_2\rangle}$  has a highest-weight vector
$x_{n_1}^{m_1}y_{n_2}^{m_2}$ of weight
$m_1\lambda_{n_1-1}-(m_1+1)\lambda_{n_1}+m_2\lambda_{n_2-1}-(m_2+1)\lambda_{n_2}$.
The arguments for these two cases are similar to the above case, and we have the same Gelfand-Kirillov dimension

\subsection{Case 2. $n_1+1<n_2=n$.}

In this case we have:
\begin{align} E_{r,i}|_{\mathscr{B}}&=-x_i\ptl_{x_r}-y_i\ptl_{y_r}\qquad  &\text{for~}& 1\leq i<r\leq
n_1,\\
 E_{s,i}|_{\mathscr{B}}&=-{x_i x_{s}}-{y_i}\ptl_{y_{s}}\qquad  &\text{for} ~&i\in\overline{1,n_1},\;s\in\overline{n_1+1,n},
\\
 E_{s,j}|_{\mathscr{B}}&=x_s\ptl_{x_j}-y_j\ptl_{y_s}\qquad   &\text{for}~& n_1<
j<s\leq n.
\end{align} Then the above $E_{j,i}$ forms a basis for the subalgebra $\mathfrak{sl}(n,\mbb{F})_-$.

From Luo-Xu \cite{Luo-Xu} we know that for  $m_1,m_2\in\bb{N}$,  ${\mathscr{H}}_{\langle
m_1,m_2\rangle}$  has a highest-weight vector $x_{n_1+1}^{m_1}y_n^{m_2}$
of weight $-(m_1+1)\lambda_{n_1}+m_1\lambda_{n_1+1}+m_2\lambda_{n-1}$. For
$m_1,m_2\in\bb{N}$ with $m_1\leq n-n_1-2$ or $m_2-m_1\leq n_1-n+1$,
${\mathscr{H}}_{\langle -m_1,m_2\rangle}$ has a highest-weight vector
$x_{n_1}^{m_1}y_n^{m_2}$ of weight
$m_1\lambda_{n_1-1}-(m_1+1)\lambda_{n_1}+m_2\lambda_{n-1}$.

The arguments for these two cases are similar to  case 1, and from proposition \ref{ak} we have the  Gelfand-Kirillov dimension equal to
$$n-1.$$
\subsection{Case 3. $n_1+1=n_2<n$.}

In this case we have:
\begin{align} E_{r,i}|_{\mathscr{B}}&=-x_i\ptl_{x_r}-y_i\ptl_{y_r}\qquad  &\text{for~}& 1\leq i<r\leq
n_1,\\
 E_{n_1+1,i}|_{\mathscr{B}}&=-{x_i x_{n_1+1}}-{y_i}\ptl_{y_{n_1+1}}\qquad  &\text{for} ~&i\in\overline{1,n_1},
\\
E_{t,i}|_{\mathscr{B}}&=-{x_i x_{t}}+y_{i}y_{t} \qquad  &\text{for} ~&i\in\overline{1,n_1},\;t\in\overline{n_1+2,n},\\
 E_{t,n_1+1}|_{\mathscr{B}}&=x_{t}\ptl_{x_{n_1+1}}+y_{n_1+1} y_{t}\qquad  &\text{for} ~&t\in\overline{n_{1}+2,n},\\
 E_{t,p}|_{\mathscr{B}}&=x_t\ptl_{x_p}+y_t\ptl_{y_p}\qquad&\text{for}~& n_1+2\leq p<t\leq
n.
\end{align} Then the above $E_{j,i}$ forms a basis for the subalgebra $\mathfrak{sl}(n,\mbb{F})_-$.

From Luo-Xu \cite{Luo-Xu} we know that for
$m_1,m_2\in\bb{N}$, ${\mathscr{H}}_{\langle-m_1,-m_2\rangle}$ has  a
highest-weight vector $x_{n_1}^{m_1}y_{n_1+2}^{m_2}$ of weight
$m_1\lambda_{n_1-1}-(m_1+1)\lambda_{n_1}-(m_2+1)\lambda_{n_1+1}+m_2(1-\delta_{n_1,n-2})\lambda_{n_1+2}$.
For $m_1,m_2\in\bb{N}$ with $m_2-m_1\geq 0$, ${\mathscr{H}}_{\langle
m_1,-m_2\rangle}$ has a highest-weight vector
$x_{n_1+1}^{m_1}y_{n_1+2}^{m_2}$ of weight
$-(m_1+1)\lambda_{n_1}+(m_1-m_2-1)\lambda_{n_1+1}+m_2(1-\delta_{n_1,n-2})\lambda_{n_1+2}$.
For $m_1,m_2\in\bb{N}$ with $m_1-m_2\geq 0$, ${\mathscr{H}}_{\langle
-m_1,m_2\rangle}$  has a highest-weight vector
$x_{n_1}^{m_1}y_{n_1+1}^{m_2}$ of weight
$m_1\lambda_{n_1-1}+(m_2-m_1-1)\lambda_{n_1}-(m_2+1)\lambda_{n_1+1}$.

The arguments for these three cases are similar to case 1, and we have the  Gelfand-Kirillov dimension equal to
$$d=\left\{
    \begin{array}{ll}
      2n-3, & \hbox{if $2=n_1+1=n_2<n$ ~or~ $n_1+1=n_2=n-1$}\\
      2n-2, & \hbox{if $2<n_1+1=n_2<n-1$.}
    \end{array}
  \right.$$


\subsection{Case 4. $n_1+1=n_2=n$.}

In this case we have:
\begin{align} E_{r,i}|_{\mathscr{B}}&=-x_i\ptl_{x_r}-y_i\ptl_{y_r}\qquad  &\text{for~}& 1\leq i<r\leq
n-1,\\
 E_{n,i}|_{\mathscr{B}}&=-{x_i x_{n}}-{y_i}\ptl_{y_{n}}\qquad  &\text{for} ~&i\in\overline{1,n-1}.
 \end{align} Then the above $E_{j,i}$ forms a basis for the subalgebra $\mathfrak{sl}(n,\mbb{F})_-$.
In this case, we denote
\eq\eta=\sum_{i=1}^{n-1}y_i\partial_{x_i}+x_ny_n.\en

From Luo-Xu \cite{Luo-Xu} we know that for $m_1,m_2\in\bb{N}$ with $m_2\leq m_1$, ${\mathscr{H}}_{\langle -m_1,m_2\rangle}$ has a highest-weight vector
$x_{n-1}^{m_1}y_n^{m_2}$ of weight
$m_1\lambda_{n-2}+(m_2-m_1-1)\lambda_{n-1}$. Moreover,
 ${\mathscr{H}}_{\langle
m,0\rangle}$ has a highest-weight vector $x_{n-1}^m$ of weight
$m\lambda_{n-2}-(m+1)\lambda_{n-1}$ for $m\in\bb{Z}$.

The arguments for these two cases are similar to  case 2, and we have the  Gelfand-Kirillov dimension equal to
$$n-1.$$

For
$m_1,m_2\in\bb{N}+1$, ${\mathscr{H}}_{\langle m_1,m_2\rangle}$  has a
highest-weight vector $\eta^{m_1+m_2}(x_{n-1}^{m_2}y_n^{-m_1})$ of
weight $m_2\lambda_{n-2}+(m_1-m_2-1)\lambda_{n-1}$.
Then similar to the the arguments in  case 2,  we have the  Gelfand-Kirillov dimension equal to
$$n-1.$$

\subsection{Case 5. $n_1=n_2<n-1$.}

In this case we have:
\begin{align} E_{r,i}|_{\mathscr{B}}&=-x_i\ptl_{x_r}-y_i\ptl_{y_r}\qquad  &\text{for~}& 1\leq i<r\leq
n_1,\\
 E_{t,i}|_{\mathscr{B}}&=-{x_i x_{t}}+y_{i}y_{t} \qquad  &\text{for} ~&i\in\overline{1,n_1},\;t\in\overline{n_1+1,n},\\
 E_{t,p}|_{\mathscr{B}}&=x_t\ptl_{x_p}+y_t\ptl_{y_p}\qquad&\text{for}~& n_1+1\leq p<t\leq
n.
\end{align} Then the above $E_{j,i}$ forms a basis for the subalgebra $\mathfrak{sl}(n,\mbb{F})_-$.

In this case, recall
\eq\zeta_1=x_{n_1-1}y_{n_1}-x_{n_1}y_{n_1-1},\;\;\zeta_2=x_{n_1+1}y_{n_1+2}-x_{n_1+2}y_{n_1+1}.
\en

From Luo-Xu \cite{Luo-Xu} we know that for $m_1,m_2\in\mbb{N}$, the subspace ${\mathscr{H}}_{\langle-m_1,-m_2\ra}$ has  a highest-weight vector
$x_{n_1}^{m_1}y_{n_1+1}^{m_2}$ of weight
$m_1(1-\dlt_{1,n_1})\lmd_{n_1-1}-(m_1+m_2+2)\lmd_{n_1}+m_2\lmd_{n_1+1}$.
The
subspace ${\mathscr{H}}_{\langle m_1+1,-m_2-m_1-1\ra}$ has  a highest-weight
vector $y_{n_1+1}^{m_2}\zeta_2^{m_1+1}$ of weight
$-(m_1+m_2+3)\lmd_{n_1}+m_2\lmd_{n_1+1}-(m_1+1)(1-\dlt_{n_1,n-2})\lmd_{n_1+2}$.
If $n_1\geq 2$, the subspace ${\mathscr{H}}_{\langle-m_1-m_2-1,m_2+1\ra}$ has
a highest-weight vector $x_{n_1}^{m_1}\zeta_1^{m_2+1}$ of weight
$(m_2+1)\lmd_{n_1-2}-m_1\lmd_{n_1-1}-(m_1+m_2+3)\lmd_{n_1}$.

Then  from Prop \ref{Nk} we have the  Gelfand-Kirillov dimension equal to
$$d=\left\{
    \begin{array}{ll}
      n-1, &  \hbox{if $1=n_1=n_2<n-1$} \\
      2n-4, & \hbox{if $n_1=n_2=2<n-1 \mathrm{~or~} 1<n_1=n_2=n-2$}\\
      2n-3, & \hbox{if $n_1=n_2=3<n=6$}\\
      2n-2, & \hbox{if $3\leq n_1=n_2\leq n-3, n\geq 7.$}
    \end{array}
  \right.$$
\subsection{Case 6. $n_1=n_2=n-1$.}

In this case we have:
\begin{align} E_{r,i}|_{\mathscr{B}}&=-x_i\ptl_{x_r}-y_i\ptl_{y_r}\qquad  &\text{for~}& 1\leq i<r\leq
n-1,\\
 E_{n,i}|_{\mathscr{B}}&=-{x_i x_{n}}+y_{i}y_{n} \qquad  &\text{for} ~&i\in\overline{1,n-1}.
\end{align} Then the above $E_{j,i}$ forms a basis for the subalgebra $\mathfrak{sl}(n,\mbb{F})_-$.

In this case, we denote
\eq\zeta_1=x_{n-2}y_{n-1}-x_{n-1}y_{n-2}.
\en

From Luo-Xu \cite{Luo-Xu} we know that for $m_1,m_2\in\mbb{N}$, the subspace ${\mathscr{H}}_{\langle-m_1,-m_2\ra}$ has a
highest-weight vector $x_{n_1}^{m_1}y_{n_1+1}^{m_2}$ of weight
$m_1(1-\dlt_{n,2})\lmd_{n-2}-(m_1+m_2+2)\lmd_{n-1}$. If $n\geq 3$,
the subspace ${\mathscr{H}}_{\langle-m_1-m_2-1,m_2+1\ra}$ has  a
highest-weight vector $x_{n_1}^{m_1}\zeta_1^{m_2+1}$ of weight
$(m_2+1)(1-\dlt_{n,3})\lmd_{n-3}-m_1\lmd_{n-2}-(m_1+m_2+3)\lmd_{n-1}$.

Then similar to the the arguments in  case 2,  we have the  Gelfand-Kirillov dimension equal to
$$n-1.$$

\subsection{Case 7. $n_1=n_2=n$.}

In this case we have:
\begin{align} E_{r,i}|_{\mathscr{B}}&=-x_i\ptl_{x_r}-y_i\ptl_{y_r}\qquad  &\text{for~}& 1\leq i<r\leq
n
\end{align} Then the above $E_{j,i}$ forms a basis for the subalgebra $\mathfrak{sl}(n,\mbb{F})_-$.

From Luo-Xu \cite{Luo-Xu} we know that the subspace ${\mathscr{H}}_{\langle-m_1-m_2,m_2\ra}$ is finite dimensional and has a
highest-weight vector $x_n^{m_1}\zeta_1^{m_2}$ of weight
$m_2(1-\dlt_{n,2})\lmd_{n-2}+m_1\lmd_{n-1}$. Obviously its Gelfand-Kirillov dimension is equal to $0$.

This completes the proof of our formula.


\section{Minimal Gelfand-Kirillov dimension module}
Let $M$ be an irreducible highest-weight module of
$\mathfrak{sl}(n,\bb F)$. We denote its Gelfand-Kirillov dimension
by $d_{M}$. From Vogan \cite{Vogan-81} and Wang \cite{wang-99}, we
know that the minimal Gelfand-Kirillov dimension is $d_M=n-1$. The
corresponding modules are called the minimal GK-dimension module.

From Luo-Xu \cite{Luo-Xu} and our formula \ref{formula}, we find the following result.
\begin{corollary}Let $n(\geq 2)$ be a positive integer. Fix $n_1,n_2\in\overline{1,n}$ with $n_1\leq n_2$.
The irreducible highest-weight
$\mathfrak{sl}(n,\bb{F})$-module ${\mathscr{H}}_{\langle \ell_1,\ell_2 \rangle}$  has the minimal Gelfand-Kirillov dimension if and only if
\begin{enumerate}
                        \item[(1)] $n_1< n_2=n$.
                           \item[(2)] $n_1=n_2=1$.
                           \item[(3)] $n_1=n_2=n-1$.
                         \end{enumerate}

\end{corollary}
%
%
%
%
%
%
%
%
%
%
%
%

We denote $\mathscr{A}=\mbb F [x_1,...,x_n]$. Fix $1\leq n_1<n$. Changing operators $\partial_{x_r}\mapsto -x_r$ and $x_r\mapsto
\partial_{x_r}$   in the canonical oscillator representation $E_{i,j}|\mathscr{A}=x_i\partial_j$ for $r\in\overline{1,n_1}$, we obtain the following
non-canonical oscillator representation of $\mathfrak{sl}(n,\bb{F})$ determined
by:
\begin{equation}E_{i,j}|_{\mathscr{A}}=\left\{\begin{array}{ll}-x_j\partial_{x_i}-\delta_{i,j}&\mbox{if}\;
i,j\in\overline{1,n_1};\\ \partial_{x_i}\partial_{x_j}&\mbox{if}\;i\in\overline{1,n_1},\;j\in\overline{n_1+1,n};\\
-x_ix_j &\mbox{if}\;i\in\overline{n_1+1,n},\;j\in\overline{1,n_1};\\
x_i\partial_{x_j}&\mbox{if}\;i,j\in\overline{n_1+1,n}.
\end{array}\right.\end{equation}
For any $k\in\bb{Z}$, we denote
$${\mathscr{A}}_{\langle
k\rangle}=\mbox{Span}\:\{x^\alpha\mid\alpha\in\bb{N}\:^n;\sum_{r=n_1+1}^n\alpha_r-\sum_{i=1}^{n_1}\alpha_i=k\}.
$$ It was presented by Howe \cite{Ho} that for
$m_1,m_2\in\bb{N}$ with $m_1>0$, ${\mathscr{A}}_{\langle -m_1\rangle}$ is an
irreducible highest-weight $\mathfrak{sl}(n,\bb{F})$-submodule with highest weight $m_1\lambda_{n_1-1}-(m_1+1)\lambda_{n_1}$ and ${\mathscr{A}}_{\langle
m_2\rangle}$ is an irreducible highest-weight $\mathfrak{sl}(n,\bb{F})$-submodule
with highest weight
$-(m_2+1)\lambda_{n_1}+m_2(1-\delta_{n_1,n-1})\lambda_{n_1+1}$.

By a similar argument with our formula, we have the following corollary.

\begin{corollary}Fix $1\leq n_1<n$.
For $m_1,m_2\in\bb{N}$ with $m_1>0$, the irreducible highest-weight module ${\mathscr{A}}_{\langle -m_1\rangle}$ and ${\mathscr{A}}_{\langle
m_2\rangle}$ have the minimal Gelfand-Kirillov dimension $(n-1).$

\end{corollary}
\bf{Acknowledgment}.\quad\rm
 The author would like to thank Prof. Xiaoping Xu for several helpful discussions and his  financial support from NSFC Grant 11171324. The author is also very grateful to the referee for detailed comments on the manuscript, which helped to improve it greatly.


\begin{thebibliography}{99}





\bibitem{Be-Br-Le} Benkart, G., Britten, D. J., Lemire, F. W.:  Modules with bounded weight multiplicities for simple Lie algebras, {\it Math. Z.}, {\bf 225}, 333-353 (1997)


\bibitem{Bo-Kr} Borho, W., Kraft, H.: \"{U}ber die Gelfand-Kirillov dimension. {\it Math. Ann.}, {\bf 220}, 1-24 (1976)



\bibitem{Br-Ho-Le} Britten, D. J., Hooper, J.,  Lemire, F. W.:   Simple $C_n$ modules with multiplicities $1$ and applications, {\it Canad. J. Phys.}, {\bf 72}, 326-335 (1994)




\bibitem{Br-Kh-Le}
Britten, D. J.,  Khomenko, O., Lemire, F. W., Mazorchuk, V.:  Complete reducibility of torsion free $C_n$-modules of finite degree, {\it J. Algebra}, {\bf 276},  129-142 (2004)

\bibitem{Br-Le} Britten, D. J.,  Lemire, F. W.:   On Modules of bounded weight multiplicities for symplectic algebras, {\it Trans. Amer. Math. Soc.}, {\bf 351}, 3413-3431 (1999)



\bibitem{Fu}
Futony, V.:  The weight representations of semisimple finite-dimensional Lie algebras, {\it Ph.D. Thesis}, Kiev University, 1987.


\bibitem{Fer}
 Fernando, S. L.:  Lie algebra modules with finite-dimensional weight spaces. I, {\it Trans. Amer. Math. Soc.}, {\bf 322}, 757-781 (1990)


\bibitem{Ge-Ki}  Gelfand, I. G.,  Kirillov, A. A.:   Sur les corps li\'{e}s aux alg\'{e}bres enveloppantes des alg\'{e}bres
de Lie, {\it Publ. Math. IHES.},  {\bf 31}, 5-19 (1966)




\bibitem{Gr-Se-06}
Grantcharov, D.,  Serganova, V.: Category of $\mathfrak{sp}(2n)$-modules with bounded weight multiplicities, {\it Mosc. Math. J.},  {\bf 6}, 119-134 (2006)

\bibitem{Gr-Se-10}
Grantcharov, D.,  Serganova, V.:  Cuspidal representations of $\mathfrak{sl}(n+1)$, {\it Adv. Math.}, {\bf 224}, 1517-1547 (2010)



\bibitem{Ho} Howe, R.:   Perspectives on invariant theory: Schur
duality, multiplicity-free actions and beyond, in: The Schur
lectures (1992) (Tel Aviv),   Israel Math. Conf.
Proc., Vol. 8, pp. 1-182, Bar-Ilan Univ., Ramat Gan, 1995



\bibitem{Ja}
Jantzen, J. C.: Einh\"{u}llende Algebren halbeinfacher Lie-Algebren. (German) [Enveloping algebras of semisimple Lie algebras], Ergebnisse der Mathematik und ihrer Grenzgebiete (3) [Results in Mathematics and Related Areas (3)], Vol. 3, Springer-Verlag, Berlin, 1983



\bibitem{Joseph-78}
Joseph, A.:
Gelfand-Kirillov dimension for the annihilators of simple quotients of Verma modules,  {\it  J. London Math. Soc.} (2), {\bf 18}, 50-60 (1978)



\bibitem{Knapp}
 Knapp, A. W.:
  Lie groups beyond an Introduction, 2nd ed.,
     Progress in Mathematics, Vol. 140,
     Birkh\"{a}user Boston, Boston, 2002

\bibitem{Kr-Le}Krause, G.,  Lenagan, T. H.:  Growth of Algebras and Gelfand-Kirillov dimension, Research
Notes in Math.  Pitman Adv. Publ. Program, Vol. 116, 1985


\bibitem{Le}
Lepowsky, J.: Generalized Verma Modules, the Cartan-Helgason Theorem, and the Harish-Chandra Homomorphism, {\it J. Algebra},  {\bf 49}, 470-495 (1977)


\bibitem{Luo-Xu}  Luo, C.,  Xu, X.: $\bb{Z}^2$-graded
oscillator representations of  $\mathfrak{sl}(n)$,  {\it Commun. Algebra,}, {\bf 41},  3147-3173 (2013)
%

\bibitem{Ma}
Mathieu, O.:   Classification of irreducible weight modules,  {\it Ann. Inst. Fourier (Grenoble)}, {\bf 50}, 537-592 (2000)


\bibitem{Me}
 Melnikov, A.:  Irreducibility of the associated varieties of simple highest weight modules in $\mathfrak{sl}(n)$, {\it C. R. Acad.
Sci. Paris S\'{e}r. I  Math.}, {\bf 316},  53-57 (1993)



\bibitem{NOTYK}
Nishiyama, K.,   Ochiai, H., Taniguchi, K.,
Yamashita, H.,   Kato, S.:
  Nilpotent orbits, associated cycles and
  Whittaker models for highest weight representations,
    {\it  Ast\'erisque}, {\bf 273}, 1-163 (2001)


\bibitem{Ri} Riordan, J.: Combinatorial identities, John Wiley \& Sons, Inc., New York-London-Sydney, 1968



\bibitem{Sun}
Sun, B.:  Lowest weight modules of $\widetilde{Sp_{2n}(\bb{R})}$ of minimal Gelfand-Kirillov dimension,  {\it J. Algebra},  {\bf 319}, 3062-3074 (2008)






\bibitem{Vogan-78}
Vogan, D.:
   Gelfand-Kirillov dimension for
  Harish-Chandra modules,
    {\it Invent.Math.}, {\bf 48}, 75-98 (1978)


\bibitem{Vogan-81}
Vogan, D.:  Singular unitary representations, in: Noncommutative harmonic analysis and Lie groups,  Lecture Notes in Mathamatics, Vol. 880, pp. 506-535, Springer, Berlin-New York, 1981


\bibitem{Vogan-91}
Vogan, D.:
   Associated varieties and unipotent representations,
    in:  Harmonic Analysis on
    Reductive Groups,  Progr. Math. Vol. 101, pp. 315--388,
       Birkh\"{a}user, Boston, 1991


\bibitem{wang-99}
Wang, W.:  Dimension of a minimal nilpotent orbit, {\it Proc. Amer. Math. Soc.}, {\bf 127},  935-936 (1999)

\bibitem{Xu-08}
Xu, X.:  Flag partial differential equations and
representations of Lie algebras,  {\it Acta. Appl. Math.}, {\bf 102}, 249--280
(2008)



\bibitem{Ya-94}  Yamashita, H.:   Criteria for the finiteness of restriction of $U(\frk g)$-modules to subalgebras and applications to Harish-Chandra modules: a study in relation to the associated varieties,
     {\it J. Funct. Anal.}, {\bf 121}, 296-329 (1994)



\bibitem{Za-Sa} Zariski, O.,  Samuel, P.: Commutative algebra,  Vol. II, Springer-Verlag, New York/Berlin(World Publishing Corporation, China), 1975



\end{thebibliography}
\end{document}